\documentclass[reqno]{amsart}
\usepackage{mathrsfs}
\usepackage{amsmath}
\usepackage{amsthm}
\usepackage{amsfonts}
\usepackage{amssymb}
\usepackage{url}
\usepackage{enumerate}
\usepackage[pdftex,bookmarks=true]{hyperref}
  \usepackage[usenames, dvipsnames]{color}
\usepackage{verbatim}

\usepackage{pdfsync}

\renewcommand\Re{{\operatorname{Re}}}

\newcommand\R{{\mathbb{R}}}

\renewcommand\P{{\mathbf{P}}}
\newcommand\E{{\mathbf{E}}}

\newcommand\T{{\mathbf{T}}}

\newcommand\Var{{\operatorname{Var}}}

\newcommand\Z{{\mathbb{Z}}}

\newcommand\al{\alpha}
\newcommand\la{\lambda}

\newcommand\Be{{\mathbf e}}

\newcommand\Bg{{\mathbf g}}

\newcommand\Bv{{\mathbf v}}

\newcommand\Bx{{\mathbf x}}
\newcommand\By{{\mathbf y}}

\newcommand\BD{{\mathbf D}}

\newcommand\BN{{\mathbf N}}

\newcommand\BT{{\mathbf T}}

\newcommand\CA{{\mathcal A}}

\newcommand\CE{{\mathcal E}}

\newcommand\CS{{\mathcal S}}

\newcommand\CU{{\mathcal U}}

\newcommand\eps{\varepsilon}

\newcommand\lang{\langle}
\newcommand\rang{\rangle}

\newcommand{\Ent}{\operatorname{Ent}}

\usepackage{fullpage}

\theoremstyle{plain}
  \newtheorem{theorem}[subsection]{Theorem}

  \newtheorem{lemma}[subsection]{Lemma}
  \newtheorem{corollary}[subsection]{Corollary}
  \newtheorem{cor}[subsection]{Corollary}

  \newtheorem{condition}{Condition}

  \newtheorem{claim}[subsection]{Claim}

\theoremstyle{definition}

\begin{document}

 \title{Exponential concentration for the number of roots of random trigonometric polynomials}
 
 % Random trigonometric polynomials: concentration of the number of roots ??

\author{Hoi H. Nguyen}
 \address{Department of Mathematics\\ The Ohio State University \\ 231 W 18th Ave \\ Columbus, OH 43210, USA}
 \email{nguyen.1261@osu.edu}
\thanks{The first author is supported by National Science Foundation grant DMS-1752345. The second author is partially supported by a US-Israel BSF grant.  This work was initiated when both authors visited the American Institute of Mathematics in 
August 2019. We thank AIM for its hospitality.}

\author{Ofer Zeitouni}
\address{Faculty of Mathematics \\ Weizmann Institute and Courant Institute, NYU \\ Rehovot 76100, Israel and NY 10012, USA}  
\email{ofer.zeitouni@weizmann.ac.il}

%\subjclass[2010]{}

\begin{abstract} 
  We show that the number of real roots of random trigonometric polynomials
  with i.i.d. coefficients, which are either bounded or 
  satisfy the logarithmic 
  Sobolev inequality, satisfies an exponential concentration of measure.
\end{abstract} 

% We show that with very high probability the number of roots of random trigonometric polynomials is concentrated around its mean. Our method relies on root repulsion and works for fairly general random coefficients. 

% I cannot come up with anything more informative than above; please feel free to modify it.

\maketitle

\section{Introduction}

Consider a random trigonometric polynomial of degree $n$
\begin{equation}\label{eqn:P_n}
P_n(x) =\frac{1}{\sqrt{n}} \sum_{k=1}^n a_k \cos(kx) + b_k \sin(kx),
\end{equation}
where $a_k, b_k$ are i.i.d. copies of a random variable $\xi$ of mean zero and variance one. Let $N_n$ denote the number of roots of $P_n(x)$ for $x\in [-\pi,\pi]$. It is known from a work of Qualls \cite{Q} that when $\xi$ is standard gaussian then
$$\E N_n = 2 \sqrt{(2n+1)(n+1)/6}.$$
By a delicate method based on the Kac-Rice formula, about ten years ago Granville and Wigman \cite{GW} showed 
 \begin{theorem}\label{thm:GW} When $\xi$ is standard gaussian, there exists an explicit constant $c_{\Bg}$ such that
$$\Var(N_n)=(c_{\Bg}+o(2)) n.$$
Furthermore,
$$\frac{N_n -\E N_n}{\sqrt{c_{\Bg}n}} \to \BN(0,1).$$
\end{theorem}
This confirms a heuristic by Bogomolny, Bohigas and Leboeuf. 
More recently, Aza\"is and Le\'on \cite{AL} provided an alternative approach 
based on the
Wiener chaos decomposition. They showed that $Y_n(t) = P_n(t/n)$ converges in certain strong sense to the stationary gaussian process $Y(t)$ of covariance $r(t)= \sin(t)/t$, from which variance and CLT can be deduced. 

% As one can show that the covariance of $r_{Y_n}(\tau)$ converges to $r(\tau)=\sin(\tau)/\tau$, and $r'_{Y_n}(\tau) \to r(\tau)'$ and $r''_{Y_n}(\tau) \to r''(\tau)$ uniformly on compact intervals avoiding roots.) One can write $Y_n(t)$ as $\int_{0}^1 \gamma_n^1(t,\la) dB_1(\la)+ \int_{0}^1 \gamma_n^2(t,\la) dB_2(\la)$ with two independent Brownian motions $B_1,B_2$ and $\gamma_n^1(t,\la)=\sum_{k=1}^n \cos(kt/n)1_{[(k-1)/n, k/n)}(\la)$ and $\gamma_n^2(t,\la)=\sum_{k=1}^n \sin(kt/n)1_{[(k-1)/n, k/n)}(\la)$. From here one can use chaos decomposition for $Y_n(.)$ and approximate it with that of $Y(.)$. 

These methods do not seem to work for other ensembles of $\xi$. Under a more general assumption, recent result by O. Nguyen and Vu \cite{ONgV} shows that  
\begin{theorem}\label{thm:NgV} Assume that $\xi$ has bounded $(2+\eps_0)$-moment for a positive constant $\eps_0$, then  there exists a constant $c>0$ such that 
$$\E N_n = (2/\sqrt{3}+O(n^{-c}))n$$
and \footnote{See \cite[Section 8]{DNN}.}
$$\Var(N_n) = O(n^{2-c}).$$
\end{theorem}
Furthermore, assuming that $|\xi|$ has finite moments of all order, under an anti-concentration estimate on $\xi$ 
of the form that
%so-called Doeblin condition 
%\corO{Is that really the name? on $\xi$ (which roughly means that 
there exists an $r>0$ and $a\in \R$ for 
which $\P( \xi \in A) \ge c \mbox{\rm Leb}(A)$ for all $A\subset B(a,r)$,
a special case of a recent result by Bally, Caramellino, and Poly \cite{BCP} regarding the number $N_n([0,\pi])$  of roots  over $[0,\pi]$ \footnote{We remark that the authors of \cite{BCP} work with roots over $[0,\pi]$.} reads as follows.
\begin{theorem}\label{thm:Doeblin} There exists a constant $c_{\Bg}'$ such that 
$$\lim_{n\to \infty} \frac{1}{n} \Var(N_n([0,\pi]))= c_{\Bg}'+ \frac{1}{30} \E(\xi^4-3).$$
\end{theorem}

%\HC{added a remark about $[0,\pi]$ and provided more precise reference to \cite[Lemma 7.2]{DNN} (we'll try to make this note available soon).}

Our goal in this note is rather different from the results above, in that we are interested in the {\it concentration} (deviation) of $N_n$ rather than the asymptotic statistics. In some way, our work is motivated by a result by Nazarov and Sodin \cite{NS} on the concentration of the number of nodal domains of random spherical harmonics, and by the exponential concentration phenomenon of the number of zeros of stationary gaussian process \cite{BDFZ}. See also \cite{GaW}. We will show the following.

\begin{theorem}\label{thm:main} Let $C_0$ be a given positive constant, and suppose that either $|\xi|$ is bounded almost surely by $C_0$, or 
that its law satisfies the logarithmic Sobolev inequality \eqref{eqn:logSobolev} with parameter $C_0$. Then there exist constants $c,c'$ such that for $\eps\geq n^{-c} $ we have that
$$\P(  |N_n - \E N_n| \ge \eps n ) \le e^{ - c' \eps^9 n}.$$
\end{theorem}

Note that in case $\xi$ is Gaussian, Theorem \ref{thm:main}
bears resemblance to \cite{BDFZ}. Note however that it is not immediate
to read Theorem \ref{thm:main} from \cite{BDFZ}, since there is no direct 
relation between the length of time interval $T$ in the latter and $n$.
%analogue of the parameter $T$ in the latter.eperiodic extension
%of $P_n$ to $\mathbb{R}$ does not have an integrable spectral density, .
It is plausible that with some effort, one could modify the proof technique 
in \cite{BDFZ} to cover this case. Our methods however are completely different and apply in particular to the Bernoulli case.

%\HC{inserted two small observations, please erase if they don't make sense.} 

We also remark that in the Gaussian case, by following \cite{EK} our result yields the following equi-distribution interpretation. Consider the curve $\gamma(x)$ on the unit sphere $S^{2n-1}$ defined by our polynomial,
$$\gamma(x)= \frac{1}{\sqrt{2n}}\big(\cos(x),\sin(x),\dots, \cos(nx), \sin(nx)\big), x\in [-\pi,\pi].$$ 
For each $x$, let $\gamma(x)_{\perp}$ be the set (known as ``great hypercircles") of vectors on $S^{2n-1}$ that are orthogonal to $\gamma(x)$. Let $\gamma_\perp$ be the region (counting multiplicities) swept by $\gamma(x)_{\perp}$ when $x$ varies in $[-\pi,\pi]$. Then $\gamma_\perp$ covers $S^{2n-1}$ {\it uniformly} in the sense that the Haar measure  of those sphere points that are covered $k$-times, where $k \notin [(2/\sqrt{3}-\eps)n, (2/\sqrt{3}+\eps)n]$, is at most $e^{ - c' \eps^9 n}$ whenever $n^{-c} \le \eps $. In another direction, our result also implies an exponential-type estimate for the persistence probability that $P_n(x)$ does not have any root (over $[-\pi, \pi]$, and hence entirely).

%\Hoi{The result is restated for more general ensembles.}

%\HQ{Is the result for the gaussian case known or worth mentioning (for instance can we read it from \cite{BDFZ})?}

Our overall method is somewhat similar to \cite{NS}, but the situation for trigonometric functions seems to be rather different compared to spherical harmonics, for instance we don't seem to have analogs of \cite[Claim 2.2]{NS} or \cite[Claim 2.4]{NS} for trigonometric polynomials. Another different aspect of our work is its universality, that the concentration phenomenon holds for many other ensembles where we clearly don't have invariance property at hands. One of the main ingredients is root repulsion, which has also been recently studied in various ensembles of random polynomials,
%recently, 
see for instance \cite{FS,DNgV,NgNgV,PSZ} among others.
%and in many other recent works. 

Finally, we remark that Theorem \ref{thm:main} can also be extended to other types of $\xi$ not necessarily bounded nor satisfying the logarithmic Sobolev inequality. For instance when $|\xi|$ has sub-exponential tail, then our method, taking $C_0=n^{\delta'}$ in Theorem \ref{thm:bounded} with an appropriate $\delta'$,
 yields a sub-exponential concentration of type $\P(  |N_n - \E N_n| \ge \eps n ) =O(e^{-(\eps n)^{\delta}})$ for some constant $0<\delta<1$. Additionally, by the same argument, for any $C>0$, if $\E (|\xi|^{C'})<\infty$ for some sufficiently large $C'$ then $\P(  |N_n - \E N_n| \ge \eps n ) = O((\eps n)^{-C})$ .

%\HC{added the above remark; please erase if it does not make sense. This is because we can choose $C_0=n^{.1}$ in Theorem \ref{thm:bounded}.}

Before concluding this section we record here a corollary from Theorem \ref{thm:NgV} which will be useful later:  for $\xi$ as in the theorem, for any $\eps>0$ we have
\begin{equation}\label{eqn:Markov}
\P(|N_n -\E N_n| \ge \eps n/2) = O(\eps^{-2}n^{-c}).  
\end{equation}

%\Hoi{So the proof idea is rather simple: (1) We first show that  a positive portion of $P_n$ is not exceptional, in a sense that it has many roots (comparable to $\E N$) over many intervals, and over these intervals the function is far from having double roots. (2) We then switch the sign of a few ($\tau^2_\eps n$) coefficients of $P_n$ (i.e. perturb $P_n$ by $g$ of normalized $L_2$-norm at most $\tau$), and then show that, for all such g, over most of these intervals the number of roots of $P_n+g$ is still above $\E N -\eps n$.}

{\bf Notation.}  We will assume $n \to \infty$ throughout the note. We write $X =
O(Y)$, 
%$Y=\Omega(X)$,
 $X \ll Y$, or $Y \gg X$ if $|X| \leq CY$
for some absolute constant $C$. The constant $C$ may depend on some parameters, in which case we write
e.g. $Y=O_\tau(X)$ if $C=C(\tau)$.  We write $X \asymp Y$ if $X \gg Y$ and $Y \gg X$.
In what follows, if not specified otherwise, all of the norms on Euclidean spaces are $L_2$-norm (i.e. $d_2(.)$ distance).

\section{Some supporting lemmas}

In this section we gather several well-known results regarding trigonometric polynomials. On the deterministic side, a useful ingredient is the classical Bernstein's inequality in $L_2(\BT)$, where $\BT=[-\pi,\pi]$. The proof
is immediate
from the orthogonality relations satisfied by the trigonometric base.
\begin{theorem}\label{thm:Bern} Let $f(x) = \sum_{k=0}^{n} a_k\cos (kx) + b_k \sin(kx), x\in \BT$. Then,
  %For any $x\in \BT$ we have
$$\int_{x\in \BT} (f'(x))^2 dx \le n^2 \int_{x \in \BT} f(x)^2 dx.$$
\end{theorem}
Another crucial inequality we will be using is the so-called large sieve inequality.
\begin{theorem}\cite[Theorem 7.7]{Iw}\cite[(1.1)]{LMN} \label{thm:largesieve} Assume that $f$ is as in Theorem \ref{thm:Bern}. Then for any $-\pi \le x_1 <x_2 <\dots <x_M\le \pi$ we have
$$\sum_{i=1}^M |f(x_i)|^2 \le \frac{2n+\delta^{-1}}{2\pi} \int_{x \in \BT} f(x)^2 dx,$$
where $\delta$ is the
minimum of the gaps between $x_i, x_{i+1}$ on the torus.
\end{theorem}

As a corollary, we obtain

\begin{cor}\label{cor:nbh} Assume that $\|f\|_{L_2(\BT)} \le \tau$.
  Then the set of $x\in \BT$ with $|f(x)| \ge \la$ or $|f'(x)| 
  \ge \la n$ is  contained in the union of
$2M$ intervals of length $2\delta$,
where $M \le  \frac{2n+\delta^{-1}}{2\pi} \frac{\tau^2}{\la^2}$.
\end{cor}
\begin{proof} Choose a maximal set of $\delta$-separated points $x_i$ for which $|f(x_i)|\ge \la$. Then by Theorem \ref{thm:largesieve} we have $M \la^2 \le  \frac{2n+\delta^{-1}}{2\pi} \tau^2$. We can apply the same argument for $f'$  where by Bernstein's inequality we have $\|f'\|_2 \le n \|f\|_2 \le n \tau$. 
\end{proof}
%\HQ{Theorem \ref{thm:largesieve} is to substitute \cite[Claim 2.2]{NS} (which is not valid for trigonometric polynomials.)}

We next introduce an elementary interpolation result (see for instance \cite[Section 1.1, E.7]{BEbook}).
\begin{lemma}\label{lemma:expansion}
Assume that a trigonometric polynomial $P_n$ has at least $m$ zeros (counting multiplicities) in an interval $I$ of length $r$. Then 
$$\max_{\theta \in I} |P_n(\theta)| \le (\frac{4e r}{m})^m \max_{x\in I} |P_n^{(m)}(x)|$$
as well as 
$$\max_{\theta \in I} |P_n'(\theta)| \le (\frac{4e r}{m-1})^{m-1} \max_{x\in I} |P_n^{(m)}(x)|.$$
Consequently, if $P_n$ has at least $m$ roots on an interval $I$ with length smaller than $(1/8e) m/n$, then for any interval $I'$ of length $(1/8e) m/n$ and $I \subset I'$ we have 
\begin{equation}\label{eqn:expand:1}
\max_{\theta \in I'} |P_n(\theta)| \le (\frac{1}{2})^m (\frac{1}{n})^m \max_{x\in I'} |P_n^{(m)}(x)|
\end{equation}
as well as 
\begin{equation}\label{eqn:expand:2}
\max_{\theta \in I'} |P_n'(\theta)| \le n \times (\frac{1}{2})^{m-1} (\frac{1}{n})^m  \max_{x\in I'} |P_n^{(m)}(x)|.
\end{equation}
\end{lemma}
\begin{proof} It suffices to show the estimates for $P_n$ because $P_n'$ has at least $m-1$ roots in $I$. For $P_n$, by Hermite interpolation using the roots $x_i$ we have that for any $\theta \in I$ there exists $x\in I$ so that
$$|P_n(\theta)| = |\frac{P_n^{(m)}(x)}{m!} \prod_i (\theta-x_i)| \le  \max_{x\in I} |P_n^{(m)}(x)| \frac{r^m}{m!}.$$
\end{proof}

%\Hoi{We need the above result for the upper bound, will double check for more precision.}

% I moded out the following because we don't seem to need it

%On the probability side we will rely on the following Salem-Zygmund inequality. 
%\begin{theorem}\cite[Chapter 6, Theorem 1]{Ka} Assume that $P_n$ are as in \eqref{eqn:P_n} with either Gaussian or Bernoulli $\xi$. Then we have 
%$$\P(|P_n|_\infty \ge Ct) \le e^{-t^2}  \mbox{ and } \P(|P_n'|_\infty \ge Ct n ) \le e^{-t^2}$$
%where $C$ is an absolute constant.
%\end{theorem}

%\Hoi{I added below some specific concentration results to be used. If you think proofs are optional, we can certainly delete the appendix. Also, here is the place where we need either boundedness or log-Sobolev.}

On the probability side, for bounded random variables we will rely on the following consequence of McDiarmid's inequality.

\begin{theorem}\label{thm:bounded} Assume that $\Bx=(x_1,\dots, x_n)$, where $x_i$ are iid copies of $\xi$ of mean zero, variance one, and $|\xi| \le C_0$ with probability one. Let $\CA$ be a set in $\R^n$. Then for any $t>0$ we have
$$\P(\Bx \in \CA) \P(d_2(\Bx, \CA) \ge t \sqrt{n}) \le 4\exp(-t^4 n/16C_0^4).$$
\end{theorem}

For random variables $\xi$ satisfying the  log-Sobolev inequality, that is so that there is a positive constant $C_0$ such that for any smooth, bounded, compactly supported functions $f$ we have %\HC{some typos corrected}
\begin{equation}\label{eqn:logSobolev}
\Ent_\xi(f^2) \le C_0 \E_\xi|\nabla f|^2,
\end{equation}
where $\Ent_\xi(f) = \E_\xi(f \log f) - (\E_\xi(f)) (\log \E_\xi(f))$, we  use the following.

\begin{theorem}\label{thm:sobolev}  Assume that $\Bx=(x_1,\dots, x_n)$, where $x_i$ are iid copies of $\xi$ satisfying \eqref{eqn:logSobolev} with a given $C_0$. Let $\CA$ be a set in $\R^n$. Then for any $t>0$ we have
$$\P\big(d_2(\Bx,A) \ge t\sqrt{n} \big) \le 2 \exp\big(-\P^2(\Bx\in \CA) t^2 n/4C_0 \big).$$
In particularly, if $\P(\Bx\in \CA)\ge 1/2$ then $\P(d_2(\Bx,A) \ge t\sqrt{n}) \le 2 \exp(-t^2 n/16C_0)$. Similarly if $\P(d_2(\Bx,A) \ge t\sqrt{n})\ge 1/2$ then $ \P(\Bx\in \CA)\le 2 \exp(-t^2 n/16C_0)$.
\end{theorem}
% Let $\CA' =\{X, d_2(X,\CA) \ge t\sqrt{n}\}$. Then for any $\Ba\in \CA, d_2(\Ba,\CA') \ge t \sqrt{n}$.  

The proofs of these well-known results will be presented in Appendix \ref{section:concentration} for completeness.

\section{Repulsion estimate}\label{section:repulsion}
%\Hoi{I have rewritten this section (and Section \ref{section:fourier}) for $\xi$ of bounded $(2+\eps_0)$-moment; the change is light because there is a way to pass from $\|.\|_\psi$ to $\|.\|_{\R/\Z}$.}

We show that the measure of $t\in [-\pi,\pi]$ where both $|P_n(t)|$ and $|P_n'(t)|$ are small is negligible. More precisely we will be working with the following condition. 
\begin{condition}\label{cond:t} Let $0<\tau\le 1/64$ be given, and let $C_0'$ be a  positive constant to be chosen sufficiently large. Assume that $t \in [-\pi,\pi]$ is such 
 that there do not exist integers $k$ with $|k| \le C_0'$ satisfying 
$$\|k t/\pi\|_{\R/ \Z} \le n^{-1+8\tau}.$$
\end{condition}
Here $\|.\|_{\R/\Z}$ is the distance to the nearest integer.
\begin{theorem}\label{thm:repulsion} Assume that $\xi$ has mean zero and variance one.
% and bounded $(2+\eps_0)$ moment.
 Then as long as $\al> 1/n$,  $\beta>1/n$ and 
$t$ satisfies Condition \ref{cond:t} with given $\tau,C_0'$ we have
$$\P\big( |P_n(t)|\le \al \wedge |P_n'(t)| \le \beta n \big) = O_{\tau, C_0'}(\al \beta).$$
\end{theorem}
In application we just choose $\al,\beta$ to be at least $n^{-c}$ for some small constant $c$. We will also choose $\tau=1/64$.
% Our bound is similar to a corresponding bound in \cite[p. 1346]{NS} for spherical harmonic functions with standard gaussian. 
Note that we can view the event in Theorem \ref{thm:repulsion} as a random walk event in $\R^2$ 
$$\frac{1}{\sqrt{2n}}\sum_{i=1}^{n} (z_i \Bv_i + z_i' \Bv_i') \in [-\al,\al] \times [-\beta ,\beta],$$ 
where $z_i, z_i'$ are iid copies of the random variables $\xi$, with  
%\HC{some small typos corrected}
$$\Bv_i := (\cos(it), -\frac{i}{n} \sin(it)) \mbox{ and } \Bv_{i}' := (\sin(it), \frac{i}{n} \cos(it)).$$

We now discuss how to prove Theorem \ref{thm:repulsion}. 

% Our treatment is similar to \cite[Lemma 4.3]{KSch} but the proof is more direct and it works for other random variables beside Bernoulli. Some more delicate estimates for one and two variables can be found in \cite{DNN}.

Given a real number $w$ and the random variable $\xi$, we define the $\xi$-norm of $w$ by
$$\|w\|_\xi := (\E\|w(\xi_1-\xi_2)\|_{\R/\Z}^2)^{1/2},$$
where $\xi_1,\xi_2$ are two iid copies of $z$. For instance if $\xi$ is $\pm 1$ Bernoulli then $\|w\|_\xi^2 = \| 2w\|_{\R/\Z}^2/2$.

Using this notation, and that $|\sin(\pi x)| \ge \|x\|_{\R/\Z}$, we can bound the characteristic function 
\begin{equation}
\label{eq-char}
\prod \phi_i(x) = \prod \E e( \xi_i \langle \Bv_i, x \rangle) \E e( \xi_i' \langle \Bv_i', x \rangle),
\end{equation}
where $e(y):=e^{iy}$,  as follows (see also \cite[Section 5]{TVcir}):
\begin{align}\label{eqn:cahr:bound}
|\prod \phi_i(x)|  = |\prod|\E e( \xi_i \langle \Bv_i, x \rangle) \E e( \xi_i' \langle \Bv_i', x \rangle) | &\le \prod_i [|\E e( \xi_i \langle \Bv_i,  x \rangle)|^2/2+1/2][|\E e( \xi_i' \langle \Bv_i',  x \rangle)|^2/2+1/2] \nonumber \\
&\le \exp(-\sum_i (\|\langle \Bv_i, x/2\pi \rangle\|_\xi^2 +  \|\langle \Bv_i', x/2\pi \rangle\|_\xi^2)/2).
\end{align}

Hence if we have a good lower bound on the exponent $\sum_i (\|\langle \Bv_i, x/2\pi \rangle\|_\xi^2+\|\langle \Bv_i', x/2\pi \rangle\|_\xi^2)$ then we would have a good control on $|\prod \phi_i(x)|$. Furthermore, by definition
\begin{align}\label{eqn:decoupling}
\sum_i (\|\langle \Bv_i, x/2\pi \rangle\|_\xi^2 + \|\langle \Bv_i', x/2\pi \rangle\|_\xi^2) &= \sum_i \E\| \langle \Bv_i, x/2\pi \rangle (\xi_1-\xi_2)\|_{\R/\Z}^2 +\sum_i \E\| \langle \Bv_i', x/2\pi \rangle (\xi_1-\xi_2)\|_{\R/\Z}^2\nonumber \\
& = \E (\sum_i \| \langle \Bv_i', x/2\pi \rangle (\xi_1-\xi_2)\|_{\R/\Z}^2 + \sum_i \| \langle \Bv_i', x/2\pi \rangle (\xi_1-\xi_2)\|_{\R/\Z}^2) \nonumber \\
&= \E_y (\sum_i \| y \langle \Bv_i, x/2\pi \rangle\|_{\R/\Z}^2+  \| y \langle \Bv_i', x/2\pi \rangle\|_{\R/\Z}^2),
\end{align}
where $y=\xi_1-\xi_2$.
As $\xi$ has mean zero and variance one,
% and bounded $(2+\eps_0)$-moment,
 there exist strictly  positive constants $c_1\le c_2,c_3$ such that $\P(c_1 \le |y/2\pi| \le c_2) \ge c_3$, and so 
\begin{equation}\label{eqn:y}
\E_y( \sum_i\| y \langle \Bv_i, x/2\pi \rangle\|_{\R/\Z}^2 +\|y \langle \Bv_i', x/2\pi \rangle\|_{\R/\Z}^2 ) \ge c_3 \inf_{c_1\le |y| \le c_2}( \sum_i \| y \langle \Bv_i, x \rangle\|_{\R/\Z}^2+ \| y \langle \Bv_i', x \rangle\|_{\R/\Z}^2).
\end{equation}

We then rely on the following estimate, whose proof will be presented in Appendix \ref{section:fourier}. 

\begin{theorem}\label{thm:fourier} Under the same assumption on $\xi$ as in Theorem \ref{thm:repulsion}, and 
with $t$ satisfying Condition \ref{cond:t} with given $\tau,C_0'$, the following holds for sufficiently large $n$. For any $x \in \R^2$ such that $n^{5\tau -1/2} \le \|x\|_2 \le n^{1-8\tau}$ we have
with the notation \eqref{eq-char}, 
 $$|\prod_i \phi_i(x)| \le e^{-n^{\tau}}.$$
\end{theorem}
%\corO{How can that be true? $r$ appears in the condition but not in the conclusion, can't one take simply $r\to 0$? }
%\HC{I deleted $r$.}
We now conclude the small ball probability estimate.
\begin{proof}(of Theorem \ref{thm:repulsion}) As we can cover the given region by disks, without loss of generality we will
consider  $\al = \beta$ and work with balls of radius $\al$. For convenience set
$$t_0:=\al^{-1}.$$ 
We can bound the small ball probability by  (see for instance \cite[Eq. 5.4]{AP} or \cite{E,H})
\begin{align*}
\P\big (\frac{1}{\sqrt{2n}} \sum_i  (\xi_i \Bv_i + \xi_i' \Bv_i') \in B(a,\al) \big ) &\le C  (\frac{n}{t_0^2}) \int_{\R^2}| \prod_i \phi_i(x)| e^{-\frac{n \|x\|_2^2}{2 t_0^2}} dx\\
&\le C' \al^2 n \int_{\R^2}i| \prod_i \phi_i(x)| e^{-\frac{n \|x\|_2^2}{2 t_0^2}}dx.
\end{align*}
We break the integral into three parts, $J_1$ when $\|x\|_2 \le r_0 =O(1)$, $J_2$ when  $r_0 \le \|x\|_2 \le R=n^{1-8\tau}$, and $J_3$ for the remaining integral. 

For $J_1$, recall from \eqref{eqn:cahr:bound} and \eqref{eqn:y} that 
\begin{align*}
|\prod \phi_i(x)| &\le \exp(-(\sum_i \|\langle \Bv_i, x \rangle\|_\xi^2 +  \|\langle \Bv_i', x \rangle\|_\xi^2)/2)\\
&\le \exp(-c_3 \inf_{c_1\le y \le c_2} ( \sum_i \| y \langle \Bv_i, x \rangle\|_{\R/\Z}^2+ \| y \langle \Bv_i', x \rangle\|_{\R/\Z}^2)).
\end{align*}
So if $\|x\|_2 \le r_0$ for sufficiently small $r_0$ then we have $\|y\langle \Bv_i,  x \rangle\|_{\R/\Z} +\|y\langle \Bv_i',  x\rangle\|_{\R/\Z} = \|y\langle \Bv_i, x \rangle\|_2+
\|y\langle \Bv_i', x \rangle\|_2 $, and so because of Condition \ref{cond:t} (see Claim \ref{claim:t},
 which implies that $\sum_i (\langle \Be, \Bv_i\rangle^2 +\langle \Be, \Bv_i'\rangle^2 )  \asymp n$ for $\Be= x/\|x\|_2$)
%\corO{I do not get this. First, what is $\Be$, but more important, condition $1$ only controls finitely many $i$s, so how do you get to $n$?} 
%\HC{I added Claim \ref{claim:t}.}
 we have that for $y>c_1$,
$$\sum_i( \|y\langle \Bv_i, x \rangle\|_{\R/\Z}^2 +\|y\langle \Bv_i', x \rangle\|_{\R/\Z}^2)/2\ge c' n\|x\|_2^2,$$
for some constant $c'=c'(c_1)$.
Thus 
\begin{align*}
J_1 &\le  C' \al^2 n \int_{\|x\|_2 \le r_0} |\prod_i \phi_i(x)| e^{-\frac{n \|x\|_2^2}{2 t_0^2}} dx \le  C' \al^2 n \int_{\|x\|_2 \le r_0}  e^{-\frac{n \|x\|_2^2}{2 t_0^2} - c'n \|x\|_2^2} du\\
& \le  C' \al^2 n \int_{\|x\|_2 \le r_0}  e^{-(\frac{n}{2 t_0^2} + c'n) \|x\|_2^2} dx  \le  C' \al^2 n  \int_{\|x\|_2 \le r_0}  e^{-(\frac{n}{2 t_0^2} + c'n) \|x\|_2^2} dx
 = O(\al^{2}).
\end{align*}

For $J_2$, recall by Theorem \ref{thm:fourier} that for $r_0 \le \|x\|_2 \le R=n^{1-8\tau}$ we have $ |\prod \phi_i(x)| \le \exp(-n^{\tau})$. Thus 
\begin{align*}
J_2 \le  C' \al^2 n  \int_{r_0 \le \|x\|_2 \le R} |\prod_i \phi_i(x)| e^{-\frac{n \|x\|_2^2}{2 t_0^2}} dx &\le C' \al^2 n \int_{r_0 \le \|x\|_2 \le R}  e^{-n^{\tau}} dx\\
& = O( C' \al^2 n^3  e^{-n^{ \tau}})= O(e^{-n^{\tau/2}}).
\end{align*}
For $J_3$, as $t_0 = \al^{-1} = O(n)$ we have
\begin{align*}
J_3 &\le C' \al^2 n  \int_{ \|x\|_2 \ge n^{1-8\tau}} |\prod_i \phi_i(x)| e^{-\frac{n \|x\|_2^2}{2 t^2}} dx = O(e^{-n^{1-16\tau}}).
\end{align*}
 \end{proof}

%\HC{changed from $u$ to $x$ for consistency.}

\section{Exceptional polynomials are rare}\label{section:exceptional} This current section is motivated by the treatment in \cite[Section 4.2]{NS}. Let $R>0$ be a sufficiently large constant. Cover $\BT$ by $\frac{2\pi n}{R}$ open interval $I_i$ of length (approximately) $R/n$ each. Let $3 I_i$ be the interval of length $3R/n$ having the same midpoint with $I_i$. Given some parameters $\al, \beta$, we call an interval $I_i$ {\it stable} for a function $f$ if there is no point in $x\in 3I_i$ such that $|f(x)|\le \al$ and $|f'(x)|\le \beta n$. Let $\delta$ be another small parameter, we call $f$ {\it exceptional} if the number of unstable intervals  is at least $\delta n$. We call $f$ not exceptional otherwise. 

For convenience, for each $P_n(x) =\frac{1}{\sqrt{n}} \sum_{k=1}^n a_k \cos(kx) + b_k \sin(kx)$ we assign a unique (unscaled) vector $\Bv_{P_n}= (a_1,\dots, a_n, b_1,\dots, b_n)$ in $\R^{2n}$, which is a random vector when $P_n$ is random. 
%\HC{For clarity, I rewrote the proofs using vectors in place of polynomials}. 
Let $\CE_e = \CE_e(R,\al,\beta; \delta)$ denote the set of vectors $\Bv_{P_n}$ associated to exceptional polynomials $P_n$. Our goal in this section is the following.
\begin{theorem}\label{thm:exceptional} Assume that $\al,\beta,\delta$ satisfy
\begin{equation}\label{eqn:parametersthm}
\al \asymp \delta^{3/2}, \beta \asymp \delta^{3/4},  \delta > n^{-2/5}.
\end{equation}
% \eqref{eqn:parameters}, then we have
Then we have
$$\P\Big(\Bv_{P_n} \in \CE_e\Big) \le e^{-c \delta^8 n},$$
where $c$ is absolute.
\end{theorem}

We now discuss the proof. First assume that $f$ (playing the role of $P_n$) is exceptional, then there are $K=\lfloor \delta n/3 \rfloor$ unstable intervals that are $R/n$-separated (and hence $4/n$-separated, as long as $R$ is chosen larger than $4$).
% \HC{corrected}.
 Now for each unstable interval in this separated family we choose $x_j \in 3 I_j$ where  $|f(x_j)|\le \al$ and $|f'(x_j)|\le \beta n$ and consider the interval $B(x_j, \gamma/n)$ for some $\gamma <1$ chosen sufficiently small (given $\delta$). Let 
$$M_j:= \max_{x\in B(x_j,\gamma/n)} |f''(x)|.$$  
By Theorem \ref{thm:largesieve} and Theorem \ref{thm:Bern} we have
$$\sum_{j=1}^K M_j^2 \le  \frac{2n+(4/n)^{-1}}{2\pi} \int_{x \in \BT} f''(x)^2 dx \le  n^5 \int_{x \in \BT} f(x)^2 dx.$$
On the other hand, in both the boundedness and the log-Sobolev cases we have $\|f\|_2 \ge 2$ exponentially small, so without loss of generality it suffices to assume $\|f\|_2 \le 2$. We thus infer from the above that the number of $j$ for which $M_j \ge C_2 \delta^{-1/2}n^2$ is at most $2 C_2^{-2} \delta n$. Hence for at least $(1/3 - 2 C_2^{-2})\delta n$ indices $j$ we must have $M_j <C_2 \delta^{-1/2}n^2$.

Consider our function over $B(x_j, \gamma/n)$, then by Taylor expansion of order two around $x_j$, we obtain for any $x$ in this interval
$$ |f(x)| \le \al + \beta \gamma + C_2 \delta^{-1/2} \gamma^2/2 \mbox{ and } |f'(x)| \le (\beta + C_2 \delta^{-1/2} \gamma) n.$$
Now consider a trigonometric polynomial $g$ such that $\|g\|_2 \le \tau$. Our polynomial $g$ has the form $g(x)= \frac{1}{\sqrt{n}} (\sum_{k=1}^{n} a_{k}' \cos (k x) + b_{k}' \sin (k x))$, where $a_{k}', b_{k}'$ are the amount we want to perturb in $f$. Then, similarly to Corollary \ref{cor:nbh}), as the intervals $B(x_j, \gamma/n)$ are $4/n$-separated, by Theorem \ref{thm:largesieve} we have
$$\sum_j \max_{x \in B(x_j, \gamma/n)} g(x)^2 \le  8n \|g\|_2^2 \le 8n \tau^2$$
and 
$$\sum_j \max_{x \in B(x_j, \gamma/n)} g'(x)^2 \le  8n \|g'\|_2^2 \le 8n^3 \tau^2.$$
Hence, again by an averaging argument, the number of intervals where either $\max_{x \in B(x_j, \gamma/n)} |g(x)| \ge C_3 \delta^{-1/2} \tau$ or   $\max_{x \in B(x_j, \gamma/n)} |g'(x)| \ge C_3 \delta^{-1/2} \tau n$ is bounded from above by $(1/3 - 2 C_2^{-2})\delta n /2$ if $C_3$ is sufficiently large. On the remaining at least $(1/3 - 2 C_2^{-2})\delta n /2$ intervals, with $h=f+g$, we have simultaneously that
$$|h(x)| \le  \al + \beta \gamma + C_2 \delta^{-1} \gamma^2/2 + C_3 \delta^{-1/2} \tau \mbox{ and } |h'(x)| \le  (\beta + C_2 \delta^{-1} \gamma + C_3 \delta^{-1/2} \tau) n.$$
% Here $f+g$ means in the Bernoulli case we have $f$, and then switch the signs of coefficients of $f$ using the position of $g$. To be made precise.
For short, let 
$$\al'= \al + \beta \gamma + C_2 \delta^{-1} \gamma^2/2 + C_3 \delta^{-1/2} \tau \mbox{ and } \beta'=\beta + C_2 \delta^{-1/2} \gamma + C_3 \delta^{-1/2} \tau.$$ 
It follows that $\Bv_h$ belongs to the set $\CU=\CU(\al, \beta,\gamma,\delta, \tau, C_1,C_2,C_3)$ in $\R^{2n}$ of the vectors  corresponding to $h$, for which the measure of $x$ with $|h(x)| \le  \al'$ and $|h'(x)| \le  \beta'n $ is at least $(1/3 - 2 C_2^{-2})\delta \gamma$ (because this set of $x$ contains $(1/3 - 2 C_2^{-2})\delta n /2$ intervals of length $2\gamma/n$). 
Putting together we have obtained the following claim. 

% \HC{reworded a bit.}

\begin{claim}
\label{claim2} Assume that $\Bv_{P_n} \in \CE_e$. Then for any $g$ with $\|g\|_2 \le \tau$ we have $\Bv_{P_n+g} \in \CU$. In other words,
$$\Big \{\Bv\in \R^{2n}, d_2(\CE_e, \Bv) \le \tau \sqrt{n}\Big \} \subset \CU.$$
\end{claim}
We next show that $\P(\Bv_{P_n} \in \CU)$ is smaller than $1/2$. Indeed, let $\BT_{e}$ denote the collection of $x \in \BT$ which can be $n^{-1+8\tau}$ approximated by rational numbers of bounded height (see Condition \ref{cond:t}, here we choose $\tau=1/64$). Thus $\BT_e$ is a union of a bounded number of intervals of length $n^{-1+8\tau}$. For each $P_n$, let $B(P_n)$ (and $B_e(P_n)$) be the measurable set of  $x\in \BT$ (or $x\in \BT_e^c$ respectively) such that $\{|P_n(x)| \le \al'\} \wedge \{|P_n'(x)| \le  \beta' n\}$. Then the Lebesgue measure of $B(P_n)$, $\mu(B(P_n))$, is bounded by $\mu(B_e(P_n)) + O(n^{-1+8\tau})$, which in turn can be bounded by 

% \HC{reworded a bit.}

$$\E \mu(B_e(P_n)) = \int_{x \in  \BT_{e}^c} \P(\{|P_n(x)| \le \al'\} \wedge \{|P'_n(x)| \le n \beta'\}) dx = O(\al' \beta'),$$
where we used Theorem \ref{thm:repulsion} for each $x$.  It thus follows that $\E \mu(B(P_n)) = O(\al' \beta') + O(n^{-1+8\tau})$. So by Markov inequality,
\begin{equation}
\label{eq-1011}
\P(\Bv_{P_n} \in \CU) \le \P\big ( \mu(B(P_n)) \ge (1/3 - 2 C_2^{-2})\delta \gamma \big ) = O(\al' \beta'/\delta \gamma) <1/2
\end{equation}
if $\al, \beta$ are as in \eqref{eqn:parametersthm} and then $\gamma, \tau$ are chosen appropriately, for instance as
%; for instance we can choose 
\begin{equation}\label{eqn:parameters}
%\al \asymp \delta^{3/2}, \beta \asymp \delta^{3/4},  \delta > n^{-2/5} \mbox{ and then }
\gamma \asymp \delta^{5/4}, \tau \asymp \delta^2. 
\end{equation}

% \Hoi{I have not tried to optimize the choices, but the method does not seem to give anything close to optimal ones.}

\begin{proof}(of Theorem \ref{thm:exceptional}) By Theorems \ref{thm:bounded} and \ref{thm:sobolev} and using Claim \ref{claim2} and \eqref{eq-1011},  we have
$$\P(\Bv_n\in \CE_e) \le e^{-c\tau^4 n}.$$
\end{proof}

\section{Roots over unstable intervals}\label{section:lowertail}

In this section we show the following lemma.

\begin{lemma}\label{lemma:manyroots'} Let $\eps$ be given as in Theorem \ref{thm:main}. Assume that the parameters $\al, \beta, \tau$ are chosen as in \eqref{eqn:parametersthm} and \eqref{eqn:parameters}, and $\delta$ is chosen such that
 \begin{equation}\label{eqn:delta:upper}
\delta \le  \frac{c_0 \eps } {\log (1/\eps)}
\end{equation}
for some small positive constant $c_0$. Assume that a trigonometric polynomial $P_n$ has at least $\eps n/2$ roots over $\delta n$ disjoint  intervals of length $R/n$. Then there is a set $A \subset \BT$ of measure at least $\frac{\eps}{1024 e}$ on which 
$$\max_{x\in A}|f(x)| \le \al \mbox{ and } \max_{x\in A} |f'(x)| \le \beta n.$$
\end{lemma}

Before proving this result, we deduce that non-exceptional polynomials cannot have too many roots over the unstable intervals.

\begin{cor}\label{cor:manyroots}   Let the parameters $\eps, \al, \beta, \tau$ and $\delta$ be as in Lemma \ref{lemma:manyroots'}, and assume that $R$ is such that
$\delta R<\eps/1024 e$. Then a non-exceptional $P_n$ cannot have more than $\eps n/2$ roots over any $\delta n$ intervals $I_i$ from Section \ref{section:exceptional}. In particularly, $P_n$ cannot have more than $\eps n/2$ roots over the unstable intervals. 
\end{cor}

\begin{proof} If $P_n$ has more than $\eps n/2$ roots over some $\delta n$ intervals $I_i$, then  Lemma \ref{lemma:manyroots'} implies the existence of 
a set $A=A(P_n)$ that intersects with the set of stable intervals (because $\eps/(1024 e) > \delta R$), so that
 $\max_{x\in A}|P_n(x)| \le \al$ and $\max_{x\in A} |P_n'(x)| \le \beta n$. However, this is impossible because for any 
 $x$ in the union of the stable intervals we have either $|P_n(x)|> \al$ or $|P_n'(x)| > \beta n$.
\end{proof}

%\Hoi{So if we have too many roots over a few small intervals, then the polynomial will be small over a non-negligible set, and hence cannot be non-exceptional.}

We now prove Lemma \ref{lemma:manyroots'}. The main idea is that if $P_n$ has too many roots over a small union of intervals, then we can use Lemma \ref{lemma:expansion} to show that $|P_n|$ and $|P_n'|$ are small over a set of non-negligible measure.

\begin{proof}(of Lemma \ref{lemma:manyroots'}) Among the $\delta n$ intervals we first throw away those of less than $\eps \delta^{-1}/4$ roots, hence there are at least $\eps n/4$ roots left. For convenience we denote the remaining intervals by $J_1,\dots, J_M$, where $M \le \delta n$,  and let $m_1,\dots, m_M$ denote the number of roots over each of these intervals respectively. 

%% \HC{typos corrected.} 

In the next step (which is geared towards the use of \eqref{eqn:expand:1} and \eqref{eqn:expand:2} of Lemma \ref{lemma:expansion}),
we  expand the intervals $J_j$ to larger intervals $\bar J_j$ 
(considered as union of consecutive closed intervals appearing at the beginning of Section \ref{section:exceptional}) of length $\lceil c m_j/R \rceil \times (R/n)$ for some small constant $c$, such as $c=1/(16e)$. Furthermore, if the expanded intervals $\bar J_{i_1}',\dots, \bar J_{i_k}'$ of $\bar J_{i_1},\dots, \bar J_{i_k}$ form an intersecting chain, then we create a longer interval
 $\bar J'$ of length $\lceil c(m_{i_1}+\dots+m_{i_k})/R \rceil \times (R/n)$, which contains them and therefore contains at least $m_{i_1}+\dots+m_{i_k}$ roots. After the merging process,
 we obtain a collection 
 %relabel the interv without loss of generality we denote the non-intersecting extended intervals by 
 $\bar J_1',\dots, \bar J_{M'}'$ with the number of roots $m_1',\dots, m_{M'}'$ respectively, so that $\sum m_i'\geq \eps n/2$.
  Note that now $\bar J_i'$ has length $\lceil cm_i'/R \rceil \times (R/n) \approx cm_i'/n$ (because $\eps \delta^{-1}$ is sufficiently large compared to $R$) and the intervals are $R/n$-separated.
 % In what follows, we drop for notational convenience
 %the superscript $'$ from both $\bar J_i'$ and $m_i'$.

Next, consider the sequence $d_l:=2^l \eps \delta^{-1}/4, l\ge 0$. We classify the sequence $\{m_i'\}$ into groups $G_l$ where %% \HC{error corrected.}
$$d_l \le m_i' < d_{l+1}.$$ 
Assume that each group $G_l$ has $k_l=|G_l|$ distinct extended intervals. As each of these intervals has between $d_l$ and $d_{l+1}$ roots, we have 
$$\sum_l k_l d_l \ge \sum_i m_i'/2 \ge \eps n/8.$$
For given $\al, \beta$, we call an index $l$ {\it bad} if 
$$(1/2)^{d_{l}}  (n/2k_{l})^{1/2}  \ge  \la= \min\{\al/4,\beta/4\}.$$ 
That is when
$$k_l \le \frac{n}{2\la^2 4^{d_{l}}}.$$
The total number of roots over the intervals corresponding to bad indices can be bounded by 
$$\sum_i m_i' \le \sum_l k_l d_{l+1}   \le \frac{n}{2\la^2} \sum_{l=0}^\infty\frac{2 d_{l}}{ 4^{d_{l}}} \le \frac{n}{\la^2 2^{\eps\delta^{-1}}} \asymp \frac{n}{\delta^3 2^{\eps\delta^{-1}}} \le \eps n/32$$
where we used the fact that $\delta \le  \frac{c_0 \eps } {\log (1/\eps)}$ for some small constant $c_0$.
% \Hoi{This is where we needed $\delta$ to be slightly smaller than $\eps$.}

Now consider a group $G_{l}$ of each good index $l$. Notice that these intervals have length approximately between $cd_l/n$ and $2cd_l/n$. Let $I$ be an interval among the $k_l$ intervals in $G_l$. By Lemma \ref{lemma:expansion} and by definition we have 
\begin{equation}\label{eqn:PP_d}
\max_{x \in I} |P_n(x)| \le (\frac{1}{2})^{d_l} (\frac{1}{n})^{d_l} \max_{x\in I} |P_n^{(d_l)}(x)| \le \frac{\la}{ (n/2k_{l})^{1/2}  } (\frac{1}{n})^{d_l} \max_{x\in I} |P_n^{(d_l)}(x)|
\end{equation}
as well as 
\begin{equation}\label{eqn:PP_d'}
\max_{x \in I} |P_n'(x)| \le n \times (\frac{1}{2})^{d_l-1} (\frac{1}{n})^{d_l}  \max_{x\in I} |P_n^{(d_l)}(x)|  \le n \times \frac{2\la}{ (n/2k_{l})^{1/2}  } (\frac{1}{n})^{d_l} \max_{x\in I} |P_n^{(d_l)}(x)|.
\end{equation}
On the other hand, as these $k_l$ intervals are $R/n$-separated (and hence $4/n$-separated), by Theorem \ref{thm:largesieve} and Theorem \ref{thm:Bern}  we have 
$$\sum_{\bar J_i'\in G_l}\max_{x\in \bar J_{i}'}(P_n^{(d_{l})}(x))^2 \le n \int_{x\in \T} (P_n^{(d_{l})}(x))^2 dx \le n \times n^{2d_l} \int_{x\in \T} (P_n(x))^2 dx \le 2 n^{2d_l+1}.$$
Hence we see that for at least half of the intervals $J_i'$ in $G_l$ 
$$\max_{x\in J_{i}'}|P^{(d_{l})}(x)| \le 4 (n/k_{l})^{1/2} n^{d_l} .$$ 
It follows from \eqref{eqn:PP_d} and \eqref{eqn:PP_d'} that over these intervals
$$\max_{x \in J_{i}'} |P_n(x)| \le  \frac{\la}{ (n/2k_{l})^{1/2}  } (\frac{1}{n})^{d_l}  4 (n/k_{l})^{1/2} n^{d_{l}}  \le  4\la$$
and similarly,
$$\max_{x \in J_i'} |P_n'(x)| \le n \times \frac{\la}{ (n/2k_{l})^{1/2}  } (\frac{1}{n})^{d_l} 4 (n/k_{l})^{1/2} n^{d_{l}}   \le 4 \la n.$$

Letting $A_l$ denote the union of all such intervals $J_{i}'$ of a given good index $l$, and letting $A$ denote the union of the $A_l$'s over all  good indices $l$, we obtain (with $\mu$ denoting Lebesgue measure)
%and the over all good indices $l$, and letting $A$ be the union of the intervals, then its Lebesgue measure is at least 
\begin{align*}
\mu(A) \ge \sum_{l, \textsf{good}} (cd_l/n) k_l/2& \ge \sum_{l, \textsf{good}} (c/4) d_{l+1}k_l/n \ge \sum_{l, \textsf{good}} (c/4) m_l/n \\
& \ge (c/4)(\eps n/8 -\eps n/32)/n \ge \frac{\eps}{1024 e}. 
\end{align*}
Finally, notice that over $A$ we have $\max_{x \in A} |P_n(x)| \le 4\la \le \al$ and $\max_{x \in A} |P_n'(x)|\le 4\la n \le \beta n$. 
\end{proof}

%% \HC{typos corrected.}

We conclude the section by a quick consequence of our lemma. For each $P_n$ that is not exceptional we let $S(P_n)$ be the collection of intervals over which $P_n$ is stable. Let $N_s(P_n)$ denote the number of roots of $P_n$ over the set $S(P_n)$ of stable intervals.

\begin{cor}\label{cor:control} With the same parameters as in  Corollary \ref{cor:manyroots}, we have
% Lemma \ref{lemma:manyroots'} we have
$$\P\Big(N_s(P_n) 1_{P_n \in \CE_e^c} \le \E N_n- \eps n\Big)=o(1)$$
and
$$\E \Big(N_s(P_n) 1_{P_n \in \CE_e^c}\Big) \ge \E N_n- 2 \eps n/3.$$
\end{cor}
\begin{proof} For the first bound, by Corollary \ref{cor:manyroots}, if $N_s(P_n) 1_{P_n \in \CE_e^c} \le \E N_n- \eps n$ then $N_n 1_{P_n \in \CE_e^c} \le \E N_n -\eps n/2$. Thus
\begin{align*}
\P\big(N_s(P_n) 1_{P_n \in \CE_e^c} \le \E N_n- \eps n\big) &\le \P\big(N_n(P_n) 1_{P_n \in \CE_e^c} \le \E N_n- \eps n/2\big) \\
&\le \P\big(\CE_e^c \wedge N_n(P_n) \le \E N_n- \eps n/2 \big) + \P(\CE_e)=o(1),
\end{align*}
where we used \eqref{eqn:Markov} and Theorem \ref{thm:exceptional}. For the second bound regarding $\E(N_s(P_n) 1_{P_n \in \CE_e^c})$ , let $N_{us}(P_n)$ denote the number of roots of $P_n$ over the set of unstable intervals. By Corollary \ref{cor:manyroots}, for non-exceptional $P_n$ 
 we have that
 $N_{us}(P_n)\le \eps n/2$, and hence trivially $\E (N_{us}(P_n) 1_{P_n \in \CE_e^c}) \le \eps n /2$. Because each $P_n$ has $O(n)$ roots, we then obtain
\begin{align*}
\E (N_s(P_n) 1_{P_n \in \CE_e^c}) &\ge \E N_n - \E (N_{us}(P_n) 1_{P_n \in \CE_e^c}) - \E (N_n 1_{P_n \in \CE_e})\\
&\ge   \E N_n - \eps n/2 - O( n\times  e^{-c\tau^4 n}) \ge \E N_n - 2\eps n/3.
\end{align*}
\end{proof}

\section{proof of the main results}

%In what follows we will use the same choice of parameters as in the previous section (via \eqref{eqn:parametersthm} and \eqref{eqn:parameters}) except that $\delta$ will be chosen slightly smaller than $\eps$ (see \eqref{eqn:delta:upper}). 

We first give a deterministic result (see also \cite[Claim 4.2]{NS}) to control the number of roots under perturbation.

\begin{lemma}\label{lemma:separation} Fix strictly 
positive numbers $\mu$ and $\nu$. Let $I=(a,b)$ be an  interval of length greater than $2\mu/\nu$, and let $f$ be a  $C^1$-function on $I$ such that at each point $x\in I$ we have either $|f(x)|> \mu$ or $|f'(x)| > \nu$. Then for each root $x_i \in I$ with $x_i-a>\mu/\nu$ and $b-x_i>\mu/\nu$ there exists an interval $I(x_i) = (a',b')$ where $f(a')f(b')<0$ and $|f(a')|=|f(b')|=\mu$,
% and $f(b')=\mu$, or $f(a')=\mu$ and $f(b')=-\mu$, 
such that $x_i \in I(x_i) \subset (x_i-\mu/\nu, x_i + \mu/\nu)$ and the intervals $I(x_i)$ over the roots are disjoint.
\end{lemma}

\begin{proof} We may and will assume that $f$ is not constant on $I$.
By changing $f(x)$ to $\la_1 f(\la_2 x)$ for appropriate $\lambda_1,\lambda_2$, it suffices to consider $\mu=\nu=1$.
For each root $x_i$, and for $0< t \le 1$ consider the interval $I_t(x_0)$ containing $x_0$ of those points $x$ where $|f(x)| < t$. We first show that for any $0<t_1,t_2\le 1$ we have that $I_{t_1}(x_1)$ and $I_{t_2}(x_2)$ are disjoint for distinct roots $x_i \in I$ satisfying the lemma's assumption. Assume otherwise, because $f(x_1)=f(x_2)=0$, there exists $x_1<x<x_2$ such that $f'(x)=0$ and $|f(x)| \le \min\{t_1,t_2\}$, and so contradicts with our assumption. We will also show that $I_1(x_0)\subset (x_0-1,x_0+1)$. Indeed, assume otherwise for instance that $x_0-1\in I_1(x_0)$, then for all $x_0-1<x<x_0$ we have $|f(x)|<1$, and so $|f'(x)|>1$ over this interval. Without loss of generality we assume $f'(x)>1$ for all $x$ over this interval. The mean value theorem would then imply that $|f(x_0-1)|= |f(x_0-1)-f(x_0)|>1$, a contradiction with $x_0-1 \in I_1(x_0)$.  As a consequence, we can define $I(x_i)=I_1(x_i)$, for which at the endpoints the function behaves as desired.  
\end{proof}

\begin{corollary}\label{cor:separation} Fix positive $\mu$ and $\nu$. Let $I=(a,b)$ be an  interval of length at least $2 \mu/\nu$, and let $f$ be a $C^1$-function on $I$  such that at each point $x\in I$ we have either $|f(x)|> \mu$ or $|f'(x)| > \nu$. Let $g$ be a function such that $|g(x)|< \mu$ over $I$. Then for each root $x_i \in I$ of $f$ with $x_i-a>\mu/\nu$ and $b-x_i>\mu/\nu$ we can find a root $x_i'$ of $f+g$ such that $x_i' \in (x_i-\mu/\nu, x_i + \mu/\nu)$, and also the $x_i'$ are distinct.
\end{corollary}

Now we prove Theorem \ref{thm:main} by considering the two tails separately.

\subsection{The lower tail}  We need to show that
\begin{equation}\label{eqn:lowertail}
\P(  N_n  \le \E N_n  - \eps n ) \le e^{-c' \eps^9 n}.
\end{equation}

With the parameters $\al,\beta, \delta, \tau,R$ chosen as in  Corollary \ref{cor:manyroots},
% Lemma \ref{lemma:manyroots'},
 consider a non-exceptional polynomial $P_n$. Let $g$ be a trigonometric polynomial with $\|g\|_2 \le \tau$, where $\tau$ is chosen as in \eqref{eqn:parameters}.  Consider a stable interval $I_j$ with respect to $P_n$ (there are at least $(\frac{2\pi}{R} -\delta)n$ such intervals). We first notice that the number of stable intervals $I_j$ over which $\max_{x \in 3I_j} |g(x)| > \al$ is at most at most $O(\delta n)$. Indeed, assume that there are $M$ such intervals $3I_j$. Then we can choose $M/6$ such intervals that are $R/n$-separated. By Theorem \ref{thm:largesieve} we have $(M/6) \al^2 \le n \tau^2$, which implies $M \le 6n (\tau \al^{-1})^2 =O(\delta n)$. From now on we will focus on the stable intervals with respect to $P_n$ on which $|g|$ is smaller than $\al$. 

%% \HC{Reworded a bit.}

By Corollary \ref{cor:separation} (applied to $I= 3I_j$ with $\mu=\al$ and $\nu =\beta n$, note that $\al/ \beta \asymp \delta^{3/4} <R$), because $\max_{x \in 3I_j} |g(x)| < \al$, the number of roots of $P_n+g$  over each interval $I_j$ is at least as that of $P_n$. Hence if $P_n$ is such that $N_n\ge \E N_n - \eps n/2$ and also $P_n$ has at least $ \E N_n - 2\eps n/3$ roots over the stable intervals, then by Corollary \ref{cor:manyroots}, with appropriate choice of the parameters, $P_n$ has at least $ \E N_n - \eps n$ roots over the stable intervals $I_j$ above where $|g| \le \al$, and hence Corollary \ref{cor:separation} implies that $P_n+g$ has at least $\E N_n-\eps n$ roots over these stable intervals $I_j$. In particularly $P_n+g$ has at least $\E N_n-\eps n$ roots over $\BT$. Let $\CU^{lower}$ be the collection of $\Bv_{P_n}$ from such $P_n$.  Then by Corollary \ref{cor:control} 
and \eqref{eqn:Markov} 
\begin{equation}\label{eqn:Uast}
\P(\Bv_{P_n} \in \CU^{lower}) \ge  1- \P\big(N_n\le \E N_n - \eps n/2\big) - \P\big(N_s(P_n) 1_{P_n \in \CE_e^c} \le \E N_n- \eps n\big)  \ge 1/2.
\end{equation}

 \begin{proof}(of Equation \eqref{eqn:lowertail}) By our application of Corollary \ref{cor:separation} above, the set $\{\Bv, d_2(\Bv,\CU^{lower}) \le \tau\sqrt{2n}\}$ is contained in the set of having at least $\E N_n - \eps n$ roots. Furthermore, \eqref{eqn:Uast} says that $\P(\Bv_{P_n} \in \CU^{lower}) \ge 1/2$. Hence by Theorems \ref{thm:bounded} and \ref{thm:sobolev} 
$$
\P( N_n  \ge \E N_n  - \eps n) \ge \P\Big(\Bv_{P_n}\in \big\{\Bv, d_2(\Bv,\CU^{lower}) \le \tau\sqrt{n}\big \}\Big) \ge 1-\exp(- c' \eps^9 n),
$$
%$ \le e^{-c\tau^4 n}$ in between.
where we used the fact that $\tau \asymp \delta^2$ from \eqref{eqn:parameters} and that $\delta$ satisfies \eqref{eqn:delta:upper}.
 \end{proof}

\subsection{The upper tail}
Our goal here is to justify the upper tail
\begin{equation}\label{eqn:uppertail}
\P(  N_n  \ge \E N_n  + \eps n ) \le e^{-c' \eps^9  n}.
\end{equation}

Let $\CU^{upper}$ denote the set of $\Bv_{P_n}$ for which  $N_n \ge E N_n+ \eps n$. By Theorem \ref{thm:exceptional} it suffices to assume that $P_n$ is non-exceptional.

% \HC{So roughly speaking if we have too many roots, then there are also two many roots over the stable intervals, and hence also all perturbations of it have more than roots than expected. On the other hand, \eqref{eqn:Markov} says that this probability is smaller than 1/2, and hence we can apply our concentration result to bound the original event.}
 
 \begin{proof}(of Equation \eqref{eqn:uppertail}) Assume that for a non-exceptional $P_n$ we have $N_n \ge \E N_n + \eps n$. Then by  Lemma \ref{lemma:manyroots'} (Corollary \ref{cor:manyroots}) the number of roots of $P_n$ over the stable intervals is at least $ \E N_n +2\eps n/3$. Let us call the collection of $\Bv_{P_n}$ of these polynomials by $\CS^{upper}$. Then argue as in the previous subsection (with the same parameters of $\al, \beta, \tau, \delta$), Corollary \ref{cor:manyroots} and Corollary \ref{cor:separation} imply that any $h=P_n+g$ with $\|g\|_2 \le \tau$ has at least $\E N_n +  \eps n /2$ roots. On the other hand, we know by \eqref{eqn:Markov} that the probability that $P_n$ belongs to this set of trigonometric polynomials is smaller than $1/2$. It thus follows by Theorems \ref{thm:bounded} and \ref{thm:sobolev} that 
$$\P(\Bv_{P_n}\in  \CU^{upper}) \le e^{-c' \eps^9  n},$$
%$ \le e^{-c\tau^4 n}$ in between.
where we again used that $\tau \asymp \delta^{2}$ and $\delta$ satisfies \eqref{eqn:delta:upper}.
\end{proof}

\appendix

\section{proof of Theorem \ref{thm:fourier}}\label{section:fourier} 
We first briefly show that the vectors $\Bv_i$ and $\Bv_i'$ with $t$ from Condition \ref{cond:t} spread out in the plane. This result was used in Section \ref{section:repulsion}, and will also be useful below.

\begin{claim}\label{claim:t} Assume that $t$ satisfies Condition \ref{cond:t} with a sufficiently large constant $C_0'$. Let $I = \{a+i, 0\le i\le L\} \subset [n]$ be any interval of length $L$ of $[n]$ with $L \ge n^{1-4 \tau}$. Then
\begin{enumerate} 
\item For any unit vector $\Be \in \R^2$ we have 
\begin{equation}\label{eqn:Be:t}
\sum_{i\in I} \langle \Be, \Bv_i \rangle^2 \asymp L^3/n^2 \mbox{ and } \sum_{i\in I} \langle \Be, \Bv_i'\rangle^2  \asymp L^3/n^2.
\end{equation}
\item For all $\eps_1,\eps_2\in \{-1,1\}$, and any positive integer $A_0 \le \sqrt{C_0'}$, there exists an $i\in I$ so that  
$$\eps_1 \sin(i A_0 t), \eps_2 \cos (i A_0t)>0.$$
\end{enumerate}
\end{claim}

\begin{proof} For the first statement, it suffices to show \eqref{eqn:Be:t} for $\Bv_i$, the treatment for $\Bv_i'$ is similar. Assume that $\Be=(x_1,x_2)$, then we can write
\begin{align*}
\sum_{i\in I} \lang \Be, \Bv_i \rang^2 &= \sum_{i\in I} (x_1 \cos(it + t_0) - x_2 (i/n) \sin(i t +t_0) )^2
\end{align*}
for some fixed $t_0$.
Clearly the sum over the diagonal terms is at least $L^3/3n^2$. For the cross term, consider
\begin{align*}
S = |\sum_{i\in I} (i/n) \cos(it+ t_0) \sin (it+t_0)| &=  |\frac{1}{2}\sum_{i\in I} (i/n) \sin ( 2 it+2 t_0)| = \frac{1}{4n}  | \frac{\partial}{\partial t}(\sum_{i\in I} \cos( 2 it + 2 t_0))| \\
&= \frac{1}{4n}   |  \frac{\partial}{\partial t} \Re[e(2 t_0) \sum_{i=0}^L e( 2 it)]|=  \frac{1}{4n}   |  \frac{\partial}{\partial t} \Re[e(2 t_0) \frac{e(2(L+1)t)-1}{e(2t)-1}]|.
\end{align*}
After some simplifications we obtain 
%\HC{It is a bit messy; roughly we have the real and complex parts, with a denominators $2-2 cos(2t) = 4sin(t)^2$, and there are some simplification in the derivatives, with lead to below with some other smaller terms O(1).}
\begin{align*}
|S| &=O\Big(\frac{L}{n}\frac{1}{|\sin t|} +  \frac{1}{n} \frac{1}{(\sin t)^2} \Big) = o(L^3/n^2),
\end{align*}
where we used the assumption $L\ge n^{1-4\tau}$ and Condition \ref{cond:t} that $\|t/\pi\|_{\R/ \Z} \ge n^{-1+8\tau}$.

Now we focus on the second part. By pigeonholing it is easy to see that if the angle sequence $\{i (A_0t) +a A_0 t, 0 \le i \le L\}$ does not occupy all four quarters of the plane, then there exists  a positive integer $k_0 =O(1)$ such that 
$$\|k_0 (A_0 t ) /\pi\|_{\R/\Z} = O(\frac{1}{L}) = O(\frac{1}{n^{1-8\tau}}).$$
This contradicts with Condition \ref{cond:t}.

%\HC{We can prove as follows: first, find by pigeonholing some $k_0=O(1)$ (i.e. absolute constant) such that $\|k_0 (A_0 t ) /\pi\|_{\R/\Z} < 1/10$, say. Let $y$ be the angle $k_0 (A_0q t ) (\mod 2\pi)$, and consider $y+a A_0 t,2y+a A_0 t,3y+a A_0 t, \dots$. If this angle sequence does not pass all the four quarters, then $ly \le 2\pi$ for all $l\le L/k_0$. It follows that $y \le 2 k_0\pi/L$.}

\end{proof}

%\HC{Proof rewritten by rescaling $\cos(t/n), \sin(t/n)$ back to $\cos(t)$ and $\sin(t)$. In the previous version we used $\cos(t/n), \sin(t/n)$, which looks rather unnatural.} \\

%\HQ{Should we move this out of the Appendix? The treatment is somewhat similar to \cite{KSch} but it is more direct and works for more general ensembles. There they just showed for $\|D\|_2 \le n^{1/2+o(1)}$ but it should hold up to $n$. In \cite{NgNgD} (where we want to show Theorem \ref{thm:Doeblin} for Bernoulli by a totally different method) we have a more local version allowing $\|D\|_2$ to grow as fast as $n^C$ for any $C$.}

We now discuss the proof of Theorem \ref{thm:fourier}.
% \HC{I just added this paragraph.}
 Our treatment is similar to \cite[Lemma 4.3]{KSch} but it is more direct and works for more general ensembles beside the Bernoulli case. Also, here we allow the parameter $\BD$ (see below) to be in the range $n^{-1/2+o(1)}\le \|\BD\|_2 \le n^{1-o(1)}$ rather than $\|\BD\|_2 \le n^{1/2-o(1)}$ as in \cite{KSch}, but this difference is minimal.

 For short, let 
$$r=r_n=n^{5\tau -1/2}.$$ 
Recall from \eqref{eqn:cahr:bound} and \eqref{eqn:y} that 
\begin{align*}
|\prod \phi_i(  x)| &\le \exp(-(\sum_i \|\langle \Bv_i, x/2\pi \rangle\|_\xi^2 +  \|\langle \Bv_i', x/2\pi \rangle\|_\xi^2)/2)\\
&\le \exp(-c_3 \inf_{c_1\le |y| \le c_2} \sum_i \| y \langle \Bv_i, x/2\pi \rangle\|_{\R/\Z}^2+\sum_i \| y \langle \Bv_i', x/2\pi \rangle\|_{\R/\Z}^2).
\end{align*}

% \Hoi{Above is the only change compared to the Bernoulli case.}

Hence for Theorem \ref{thm:fourier} it suffices to show that for any $\BD=(d_1,d_2)$ (which plays the role of $(y/2\pi)x$) such that $c_1r \le \|\BD\|_2\le c_2 n^{1-8\tau}$ we have 
\begin{equation}\label{eqn:thm:fourier}
\sum_i \|\langle \Bv_i,   \BD \rangle\|_{\R/\Z}^2 +  \|\langle \Bv_i',   \BD \rangle\|_{\R/\Z}^2 \ge n^{\tau}.
\end{equation}

For convenience, let  
 \begin{equation}\label{eqn:psi1}
 \psi_i = d_1 \cos(it) -d_2 \frac{i}{n} \sin (it) \mbox{ and } \psi_i' = d_1 \sin(it/n) +d_2 \frac{i}{n} \cos (it) .
\end{equation}
In other words,
$$\psi_i = \langle \BD, \Bv_i \rangle \mbox{ and } \psi_i' = \langle \BD, \Bv_i' \rangle.$$
with

Let $\Be$ be the unit vector in the direction of $D$, $\Be = \frac{\BD}{\|\BD\|_2}$. Our key ingredient in the proof of Theorem \ref{thm:fourier} is the following.

\begin{lemma}\label{lemma:differencing} Suppose that $c_1r \le \|\BD\|_2 \le c_2 n^{1-8\tau}$ and
$$|\{j \in [n]: \|\psi_j\|_{\R/\Z} + \|\psi_j'\|_{\R/\Z} > n^{-\tau}\}|\le n^{3\tau}.$$
Then for large $n$ there exists an interval $J \subset [n]$ of length $n^{1-6\tau}$ so that 
$$\sum_{j\in J} |\langle \Bv_j, \Be \rangle|^2 +  |\langle \Bv_j', \Be \rangle|^2 \ge n^{1-8\tau} \mbox{ and } \sup_{j\in J} |\psi_j| + \psi_j'| \le n^{-\tau}.$$ 
\end{lemma}

\begin{proof}(of Equation \eqref{eqn:thm:fourier}) Recall that $\Be$ is the unit vector $\BD/\|\BD\|_2$. If $|\{j \in [0,n) \cap \Z: \|\psi_j\|_{\R/\Z} > n^{-\tau}\}|\ge n^{3\tau}$ then we have 
$$\sum_i \|  \langle \Bv_i,   \BD \rangle\|_{\R/\Z}^2 + \sum_i \|  \langle \Bv_i',   \BD \rangle\|_{\R/\Z}^2  \ge n^{-2 \tau} n^{3\tau} = n^{\tau}.$$
Assume otherwise, we write $\sum_i \|\langle \Bv_i,   \BD \rangle\|_{\R/\Z}^2+\sum_i \|\langle \Bv_i',   \BD \rangle\|_{\R/\Z}^2   = \sum_i \| \|\BD\|_2 \langle \Bv_i,   \Be \rangle\|_{\R/\Z}^2+\sum_i \| \|\BD\|_2 \langle \Bv_i',   \Be \rangle\|_{\R/\Z}^2$. Then by Lemma \ref{lemma:differencing}, there exists an interval $J \subset [0,n)$ so that 
$$\sum_{j\in J} |\langle \Bv_j, \Be \rangle|^2 + |\langle \Bv_j', \Be \rangle|^2 \ge n^{1-8\tau} \mbox{ and } \sup_{j\in J} |\psi_j| + |\psi_j'| \le n^{-\tau}.$$ 
Then as for these indices $|\psi_j|= \|\psi_j\|_{\R/\Z}$ and $|\psi_j'|= \|\psi_j'\|_{\R/\Z}$ we have
\begin{align*}
\sum_i \|\langle \Bv_i,   \BD \rangle\|_{\R/\Z}^2+ \sum_i \|\langle \Bv_i',   \BD \rangle\|_{\R/\Z}^2  & \ge \sum_{j\in J} \|\langle \Bv_j,   \BD \rangle\|_{\R/\Z}^2 + \|\langle \Bv_j',   \BD \rangle\|_{\R/\Z}^2 = \sum_{j\in J} \|\langle \Bv_j,   \BD \rangle\|_2^2 +  \|\langle \Bv_j',   \BD \rangle\|_2^2\\
&= \|\BD\|_2^2 \sum_{j\in J} \langle \Bv_j, \Be\rangle ^2 + \langle \Bv_j', \Be\rangle ^2 \ge (c_1r)^2 n^{1-8\tau} \ge n^{2\tau}.
\end{align*}
\end{proof}

\begin{proof}(of Lemma \ref{lemma:differencing}) We decompose into several steps. First recall that 
$$|\{j \in [0,n) \cap \Z: \|\psi_j\|_{\R/\Z}+\|\psi_j'\|_{\R/\Z} > n^{-\tau}\}|\le n^{4\tau} \mbox{ and } \sum_{j\in [n]}  |\langle \Bv_j, \Be \rangle|^2 +  |\langle \Bv_j', \Be \rangle|^2\ge c' n.$$
Divide $[n]$ into $n^{4 \tau}$ disjoint intervals $J_i$ of length $n^{1-4\tau}$ each. For each $i$, define 
$$s_i := \sum_{j\in J_i}  |\langle \Bv_j, \Be \rangle|^2+ |\langle \Bv_j', \Be \rangle|^2.$$
Then we trivially have $s_i \le |J_i| \le n^{1-4\tau}$, and $\sum_{i\le n^{4\tau}} s_i\ge n^{1-\tau}$.  Let $x$ be the number of intervals with $s_i$ larger than $n^{1-8\tau}$. Then we have 
 $$x n^{1-4\tau} + (n^{4\tau}-x) n^{1-8\tau} \ge  c' n.$$
 It follows that
 $$x \ge \frac{c'n - n^{1-4\tau}}{n^{1-4\tau} - n^{1-8\tau}} > n^{3 \tau}.$$
 As such, we have found an interval $J=J_i$ of length $n^{1-4\tau}$ in $[n]$ for which 
$$\sup_{j\in J} |\langle \Bv_j, \Be \rangle|^2 + |\langle \Bv_j', \Be \rangle|^2 \ge n^{1-8\tau}$$
and for all $j\in J$ we have 
$$\|\psi_j\|_{\R/\Z} + \|\psi_j'\|_{\R/\Z} \le n^{-\tau}.$$
Our goal is to show that for $j\in J$ we indeed have 
\begin{equation}\label{eqn:psi:RZ}
\|\psi_j\|_{\R/\Z}= |\psi_j| \mbox { and }  \|\psi_j'\|_{\R/\Z} = |\psi_j'|.
\end{equation}
This would then automatically imply Lemma \ref{lemma:differencing} with $J$ as above. In what follows, without loss of generality we just show $\|\psi_j\|_{\R/\Z}= |\psi_j|$, the treatment for $\psi'$ is similar.

{\bf Differencing.} For short let $A:=\lfloor \sqrt{C_0'} \rfloor$ (where we recall that $C_0'$ is chosen sufficiently large in Condition \ref{cond:t}). By pigeonholing we can find $p_0\in \Z, p_0 \neq 0$ and $t_0$ so that  
\begin{equation}\label{eqn:approx:pq}
p_0 \frac{t}{2\pi} -t_0 \in \Z, 1\le |p_0| \le A, |t_0| \le \frac{4}{A}.
\end{equation}
From the approximation we infer that 
\begin{equation}\label{eqn:closeone}
|e^{\sqrt{-1} p_0 t} -1|= |e^{-\sqrt{-1} (2\pi t_0)}-1| \le |2\sin(\pi t_0)| \le 4 \pi/A.
\end{equation}

Next consider 
$$\Delta^l \psi_{j+ l p_0 } = \sum_{i=0}^k \binom{k}{i} (-1)^i \psi_{j+ (i+l)p_0}.$$
 Let $m_j$ be the integer closest to $\psi_j$, then for $j\in J$ we have $|\psi_j - m_j| \le n^{-\tau}$. Applying the argument in \cite[Lemma 4.3]{KSch} (with $\phi(j/n) = j/n$) we will show
\begin{lemma}\label{lemma:diff} We have
\begin{equation}\label{eqn:Delta:diff}
|\Delta^k m_{j+lp_0}|  \le 4\|\BD\|_2 \frac{(4\pi)^{k}}{A^{(k-3)/2}} + 4 \times 2^{k}n^{-\tau}
\end{equation}
provided that $[j+ l p_0, j+(l+k) p_0 ] \subset J$.
\end{lemma}
\begin{proof}(of Lemma \ref{lemma:diff}) Recall that ${\psi}_j =d_1 \cos(j t) -d_2 \frac{j}{n} \sin (j t)$ and $\|\psi_j\|_{\R/\Z} \le n^{-\tau}$ over all $j\in J$. Consider 
$$\Delta^k {\psi}_{j+ l p_0} = \sum_{i=0}^k \binom{k}{i} (-1)^i {\psi}_{j+ (i+l)p_0}.$$
We first have
\begin{align*}
\sum_{i=0}^k \binom{k}{i} (-1)^i  \cos( i p_0 t + (j+ l p_0 ) t) &=\Re \sum_{i=0}^k \binom{k}{i} (-1)^i  e^{\sqrt{-1} [i p_0 t + (j+ l p_0 ) t]}\\
&= \Re \Big((\sum_{i=0}^k \binom{k}{i} (-1)^i  e^{\sqrt{-1} i p_0 t}) e^{\sqrt{-1}(j+ l p_0 ) t}\Big)\\
&= \Re \Big((1-e^{\sqrt{-1} p_0 t})^k e^{\sqrt{-1}(j+ l p_0 ) t}\Big) \le (4 \pi/A)^k
\end{align*}
where we used \eqref{eqn:closeone} in the last estimate. It also follows that
\begin{align*}
\sum_{i=0}^k \binom{k}{i} (-1)^i  \frac{j+(i+l)p_0}{n}\sin( i p_0 t + (j+ l p_0 ) t) &=\frac{1}{n}\frac{\partial}{\partial t}\Big(\sum_{i=0}^k \binom{k}{i} (-1)^i  \cos( i p_0 t + (j+ l p_0 ) t)\Big)\\
&=\frac{1}{n} \frac{\partial}{\partial t} \Big(\Re \big((1-e^{\sqrt{-1} p_0 t})^k e^{\sqrt{-1}(j+ l p_0 ) t}\big)\Big)\\
&=\frac{1}{n} \Re \Big( \frac{\partial}{\partial t} \big((1-e^{\sqrt{-1} p_0 t})^k e^{\sqrt{-1}(j+ l p_0 ) t}\big)\Big)\\
&= \Re \Big( -\sqrt{-1} \frac{kp_0}{n} \big((1-e^{\sqrt{-1} p_0 t})^{k-1} e^{\sqrt{-1}(j+ l p_0 ) t}\big) \\
&+ \sqrt{-1}\frac{j+ l p_0}{n} \big((1-e^{\sqrt{-1} p_0 t})^k e^{\sqrt{-1}(j+ l p_0 ) t}  \Big)\\
&\le A (4 \pi/A)^{k-1} + (4 \pi/A)^k < (4\pi)^{k}/A^{k-3}.
\end{align*}
Putting the bounds together we obtain
$$|\Delta^k {\psi}_{j+ l p_0}| = |\sum_{i=0}^k \binom{k}{i} (-1)^i {\psi}_{j+ (i+l)p_0}| \le (|d_1|+|d_3|) (\frac{4 \pi}{\sqrt{A}})^k + (|d_2|+|d_4|)\frac{(4\pi)^{k}}{A^{k-3}} < 4\|\BD\|_2 \frac{(4\pi)^{k}}{A^{k-3}}.$$
It thus follows that 
\begin{align*}
|\Delta^k m_{j+lp_0}| &\le |\Delta^k {\psi}_{j+ l p_0 }|  + |\Delta^k ({\psi}_{j+ l p_0 } - m_{j+ l p_0})| \le 4\|\BD\|_2 \frac{(4\pi)^{k}}{A^{k-3}} + 4 \times 2^{k}n^{-\tau}.
\end{align*}
\end{proof}

Note that if we choose $k =k_0 = \lfloor c \log_2 n \rfloor$  for some small constant $c$ (such as $c<\tau/2$), and then as $A$ is  a sufficiently large constant, the RHS of \eqref{eqn:Delta:diff} is smaller than one. Because these numbers are integer, it follows that as long as  $[j+ l p_0, j+(l+k_0) p_0] \subset J$ we must have 
$$\Delta^{k_0} m_{j+lp_0}=0.$$
It follows from \eqref{eqn:Delta:diff} that $m_{j+ l p_0 } = P_j(l)$ where $P_j$ is a real polynomial of degree at most $k_0-1$. 

{\bf Vanishing integral part.} We next show that $P_j$ is a constant. Indeed, assuming otherwise, then as $P'_j$ has at most $k_0-2$ roots, there is an interval of length $|J|/k_0$ where $P_j$ is strictly monotone. But on this interval (of length of order $n^{1-4 \tau -o(1)}$ at least), $m_j \in [-n^{1-8\tau}, n^{1-8\tau}]$ (because $|m|\le \|\BD\|_2\le n^{1-8\tau}$), so this is impossible. Thus we have shown that 
$$m_{j+ l p_0 } = m_j \mbox{ for all } j,l\in \Z \mbox{ such that }  [j+ l p_0, j+(l+k_0) p_0 ] \subset J=[a,b].$$
Note that for any fixed $j$, the range for $l$ is $(a-j)/p_0 \le l \le (b-n^{o(1)}-j)/p_0$, which is an interval of length of order $n^{1-4\tau}$. Over this range of $l$, and with $A_0=p_0\le A = \lfloor \sqrt{C_0'} \rfloor$, the condition of $t$ in Condition \ref{cond:t} (see Claim \ref{claim:t}) implies that $ \psi_{j+lA_0} =d_1 \cos((j+ l A_0)t) -d_2 \frac{i}{n} \sin ((j+ l A_0) t)$ changes sign. But as $m_{j+ l A_0}=m$ is the common integral part for all $l$, this is impossible unless $m=0$. This completes the proof of \eqref{eqn:psi:RZ}.
%\HC{We used Claim \ref{claim:t} for the sign change.}
\end{proof}

\section{Concentration results}\label{section:concentration}
\begin{proof}(of Theorem \ref{thm:bounded}) Consider the function $F(\Bx):= d_1(\Bx,\CA)$, which measures the $L_1$-distance. This function is $2C_0$-Lipschitz (coordinatewise), so by  McDiarmid's inequality, with $\mu=\E F(\Bx)$
$$\P(|F(\Bx) - \mu | \ge \la) \le 2 \exp(- \la^2/2nC_0^2).$$
This then implies that 
$$P(F(\Bx) =0) \P(F(\Bx) \ge \la) \le 4 \exp(-\la^2/4nC_0^2).$$ 
Indeed, if $\la \le \mu$ then 
$$\P(F(\Bx)=0) \le \P(F(\Bx) -\mu \le -\mu) \le 2\exp(-\mu^2/2nC_0^2) \le 2\exp(-\la^2/2nC_0^2),$$ 
while if $\la \ge \mu$ then 
\begin{align*}
\P(F(\Bx) =0) \P(F(\Bx) \ge \la) &\le \P(F(\Bx) -\mu \le -\mu)  \P(F(\Bx) -\mu \ge \la -\mu) \\
& \le  4 \exp(-(\mu^2 +(\la-\mu)^2)/2nC_0^2) \le 4 \exp(-\la^2/4nC_0^2).
\end{align*}
Now because of boundedness, $\|\Bx-\By\|_2^2 \le 2C_0 \|\Bx-\By\|_1$. So if $d_2(\Bx,\CA) \ge t \sqrt{n}$ then $d_1(\Bx,\CA) \ge t^2 n/2C_0$. We thus obtain
$$\P(\Bx \in \CA) \P(d_2(\Bx, \CA) \ge t \sqrt{n}) \le \P(\Bx \in \CA) \P(d_1(\Bx, \CA) \ge t^2 n/2C_0) \le 4\exp(-t^4 n/16C_0^4).$$
\end{proof}

\begin{proof}(of Theorem \ref{thm:sobolev}) Let $\la:= t\sqrt{n}$ and $F(\Bx) := \min\{d_2(\Bx,\CA), \la\}$. Then $F$ is 1-Lipschitz, and 
$$\E F(\Bx) \le (1- \P(\Bx\in \CA))\la.$$
It is known (see for instance \cite{L}) that for  distributions satisfying log-Sobolev inequality we have that 
$$\P(F(\Bx) \ge \E F(\Bx) + t) \le \exp(-t^2/4C_0).$$
Thus, since $\E F(\Bx)=\P(\Bx\not\in \CA)\E(F(\Bx)|\Bx\not\in \CA)\leq \lambda \P(\Bx\not\in \CA)$,
$$\P(d_2(\Bx,A) \ge \la) = \P(F(\Bx) \ge \la) \le \P(F(\Bx) \ge \E F(\Bx) + \P(\Bx\in \CA) \la) \le \exp(-\P^2(\Bx\in \CA) \la^2/4C_0).$$
\end{proof}

{\bf Acknowledgements.}  The first author is grateful to O. Nguyen for helpful discussion around Theorem \ref{thm:NgV}, and to T. Erd\'ely for help with references.

\nocite{*}

\end{document}